\documentclass[10pt]{amsart}
     \makeatletter
     \def\section{\@startsection{section}{1}%
     \z@{.7\linespacing\@plus\linespacing}{.5\linespacing}%
     {\bfseries
     \centering
     }}
     \def\@secnumfont{\bfseries}
     \makeatother
\newtheorem{theorem}{Theorem}[section]
\newtheorem{lemma}[theorem]{Lemma}

\newtheorem{corollary}[theorem]{Corollary}
\theoremstyle{definition}

\theoremstyle{remark}

\numberwithin{equation}{section}
\usepackage{amsmath,amsthm,amssymb,amsbsy}

\newcommand{\RR}{\mathbb{R}}

\newcommand{\EEE}{\mathcal{E}}
\newcommand{\FFF}{\mathcal{F}}
\newcommand{\GGG}{\mathcal{G}}

\newcommand{\df}[1]{\,\mathrm{d}#1}                         

\newcommand{\eps}{\varepsilon}                              

\begin{document}

\setlength{\parindent}{0cm}
\setlength{\parskip}{0.5cm}

\title[Optimal Novikov-type criteria]{Optimal Novikov-type criteria
  for local martingales with jumps}

\author{Alexander Sokol}

\address{Alexander Sokol: Institute of Mathematics, University of
  Copenhagen, 2100 Copenhagen, Denmark}
\email{alexander@math.ku.dk}
\urladdr{http://www.math.ku.dk/$\sim$alexander}



\begin{abstract}
We consider local martingales $M$ with jumps larger than $a$ for some
$a$ larger than or equal to $-1$, and prove Novikov-type criteria for
the corresponding exponential local martingale
to be a uniformly integrable martingale. We obtain criteria using both
the quadratic variation and the predictable quadratic variation. We
prove optimality of the coefficients in the criteria. As a corollary,
we obtain a verbatim extension of the classical Novikov criterion for
continuous local martingales to the case of local martingales with nonnegative jumps.
\end{abstract}

\maketitle

\noindent

\section{Introduction}
\label{sec:intro}

The motivation of this paper is the question of when an
exponential local martingale is a uniformly integrable
martingale. Before introducing this problem, we fix our notation and
recall some results from stochastic analysis.

Assume given a filtered probability space $(\Omega,\FFF,(\FFF_t)_{t\ge0},P)$
satisfying the usual conditions, see \cite{PP} for the definition
of this and other probabilistic concepts such as being a local martingale, locally
integrable, locally square-integrable, and for the quadratic variation
and quadratic covariation et cetera. For any local martingale
$M$, we say that $M$ has initial value zero if $M_0=0$. For any local
martingale $M$ with initial value zero, we denote by $[M]$ the
quadratic variation of $M$, that is, the unique increasing adapted
process with initial value zero such that $M^2-[M]$ is a local martingale. If $M$ furthermore is locally square
integrable, we denote by $\langle M\rangle$ the predictable quadratic
variation of $M$, which is the unique increasing predictable process
with initial value zero such that $[M]-\langle M\rangle$ is a local
martingale.

For any local martingale with initial value zero,
there exists by Theorem 7.25 of \cite{HWY} a unique decomposition
$M=M^c+M^d$, where $M^c$ is a continuous local martingale and $M^d$ is
a purely discontinuous local martingale, both with initial value
zero. Here, we say that a local martingale with initial value zero is
purely discontinuous if it has zero quadratic covariation with
any continuous local martingale with initial value zero. We refer to
$M^c$ as the continuous martingale part of $M$, and refer to $M^d$ as
the purely discontinuous martingale part of $M$.

Let $M$ be a local martingale with
initial value zero and $\Delta M\ge-1$. The exponential martingale of $M$, also known as the
Dol{\'e}ans-Dade exponential of $M$, is the unique c\`{a}dl\`{a}g
solution in $Z$ to the stochastic differential equation $Z_t = 1 +
\int_0^t Z_{s-}\df{M}_s$, given explicitly as
\begin{align}
  \EEE(M)_t &= \exp\left(M_t-\frac{1}{2}[M^c]_t\right)\prod_{0<s\le t}(1+\Delta M_s)\exp(-\Delta M_s),
\end{align}
see Theorem II.37 of \cite{PP}. Applying Theorem 9.2 of
\cite{HWY}, we find that $Z$ always is a local martingale with initial
value one. Also, $\EEE(M)$ is always nonnegative. We wish to
understand when $\EEE(M)$ is a uniformly integrable martingale.

The question of when $\EEE(M)$ is a uniformly integrable
martingale has been considered many times in the litterature, and is
not only of theoretical interest, but has several applications in connection
with other topics. In particular, exponential martingales are of use
in mathematical finance, where checking uniform integrability of
a particular exponential martingale can be used to prove absence of
arbitrage and obtain equivalent martingale measures for option
pricing. For more on this, see \cite{PS} or chapters 10
and 11 of \cite{TB}. Also, exponential martingales arise naturally in
connection with maximum likelihood estimation for stochastic
processes, where the likelihood viewed as a stochastic process often
is an exponential martingale which is a true martingale, see for example the likelihood for
parameter estimation for Poisson processes given in (3.43)
of \cite{KA} or the likelihood for parameter estimation for diffusion
processes given in Theorem 1.12 of \cite{YK}. Finally, exponential martingales which are true martingales can be used in the
explicit construction of various probabilistic objects, for example
solutions to stochastic differential equations, as in Section 5.3.B of
\cite{KS2}.

Several sufficient criteria for $\EEE(M)$ to be a uniformly integrable
martingale are known. First results in this regard were obtained by \cite{AAN}
for the case of continuous local martingales. Here, we are interested in
the case where the local martingale $M$ is not necessarily
continuous. Sufficient criteria for $\EEE(M)$ to be a uniformly
integrable martingale in this case have been obtained by
\cite{LM}, \cite{ISS}, \cite{TO}, \cite{JAY} and \cite{KS}.

We now explain the particular result to be obtained in this paper. In
\cite{AAN}, the following result was obtained: If $M$ is a continuous
local martingale with initial value zero and $\exp(\frac{1}{2}[M]_\infty)$ is integrable, then
$\EEE(M)$ is a uniformly integrable martingale. This criterion is
known as Novikov's criterion. We wish to understand whether
this result can be extended to local martingales which are not
continuous.

In the case with jumps, another process in addition to the
quadratic variation process is relevant: the predictable quadratic
variation. As noted earlier, the predictable quadratic variation is defined for any
locally square-integrable local martingale $M$ with initial value zero, is denoted $\langle
M\rangle$, and is the unique predictable, increasing and locally
integrable process with initial value zero such that $[M]-\langle M\rangle$ is a local
martingale, see p. 124 of \cite{PP}. For a continuous local
martingale $M$ with initial value zero, we have that $M$ always is locally square integrable
and $\langle M\rangle =[M]$.

Using the predictable quadratic variation, the following result is
demonstrated in Theorem 9 of \cite{PS}. Let
$M$ be a locally square integrable local martingale with initial value
zero and $\Delta M\ge-1$. It then holds that if
$\exp(\frac{1}{2}\langle M^c\rangle_\infty+\langle M^d\rangle_\infty)$ is
integrable, then $\EEE(M)$ is a uniformly integrable martingale. This is an extension of the classical Novikov criterion
of \cite{AAN} to the case with jumps. \cite{PS} also argue in Example
10 that the constants in front of $\langle M^c\rangle$ and
$\langle M^d\rangle$ are optimal, although their argument contains a
flaw, namely that the formula (28) in that paper does not hold.

In this paper, we specialize our efforts to the case where
$M$ has jumps larger than or equal to $a$ for some $a\ge-1$ and prove
results of the same type, requiring either that $M$ is a locally square
integrable local martingale and that $\exp(\frac{1}{2}\langle
M^c\rangle+\alpha(a)\langle M^d\rangle)$ is integrable for some
$\alpha(a)$, or that $\exp(\frac{1}{2}[M^c]+\beta(a)[M^d])$ is
integrable for some $\beta(a)$. For all $a\ge-1$, we identify the optimal
value of $\alpha(a)$ and $\beta(a)$, in particular giving an argument circumventing
the problems of Example 10 in \cite{PS}. Our results are stated as
Theorem \ref{theorem:OptimalNovikov} and Theorem
\ref{theorem:OptimalNovikov2}. In particular, we obtain that for local
martingales $M$ with initial value zero and $\Delta M\ge0$, $\EEE(M)$ is a uniformly
integrable martingale if $\exp(\frac{1}{2}[M]_\infty)$ is integrable
or if $M$ is locally square integrable and $\exp(\frac{1}{2}\langle
M\rangle_\infty)$ is integrable, and we obtain that both the constants
in the exponents and the requirement on the jumps of $M$ are
optimal. This result is stated as Corollary
\ref{coro:NovikovExtension} and yields a verbatim extension of the
Novikov criterion to local martingales $M$ with initial value zero and
$\Delta M\ge0$.

\section{Main results and proofs}
\label{sec:main}

In this section, we apply the results of \cite{LM} to obtain optimal
constants in Novikov-type criteria for local martingales with jumps. For $a>-1$ with $a\neq0$, we define
\begin{align}
  \alpha(a)&=\frac{(1+a)\log(1+a)-a}{a^2}\quad\textrm{ and }\\
   \beta(a)&=\frac{(1+a)\log (1+a)-a}{a^2(1+a)},
\end{align}
and put $\alpha(0)=\beta(0)=\frac{1}{2}$ and $\alpha(-1)=1$. The
functions $\alpha$ and $\beta$ will yield the optimal constants in the
criteria we will be demonstrating. Before proving our main results, Theorem
\ref{theorem:OptimalNovikov} and Theorem
\ref{theorem:OptimalNovikov2}, we state three lemmas.

\begin{lemma}
\label{lemma:alphaFun}
The functions $\alpha$ and $\beta$ are continuous, positive and
strictly decreasing. Furthermore, $\beta(a)$ tends to infinity as $a$ tends to
minus one.
\end{lemma}
\begin{proof}
We first prove the result on $\alpha$. 
Define $h(a)=(1+a)\log(1+a)-a$ for $a>-1$ and $h(-1)=1$. Note that $h$ is differentiable with
$h'(a)=\log(1+a)$. By the l'H{\^o}pital
rule, we have
\begin{align}
  \lim_{a\to -1}h(a)
  &=1+\lim_{a\to-1}\frac{\log(1+a)}{(1+a)^{-1}}
   =1-\lim_{a\to-1}\frac{(1+a)^{-1}}{(1+a)^{-2}}=1,
\end{align}
which yields that $h$ and $\alpha$ are continuous at $-1$. Similarly,
\begin{align}
  \lim_{a\to0}\alpha(a)
  &=  \lim_{a\to0}\frac{\log(1+a)}{2a}
  =  \lim_{a\to0}\frac{1}{2(1+a)}=\frac{1}{2},
\end{align}
so $\alpha$ is continuous at $0$. As $h$ is zero at zero, $h(a)$ is positive for
$a\neq0$, from which is follows that $\alpha$ is positive. It remains
to show that $\alpha$ is strictly decreasing. For $a\ge-1$ with
$a\notin\{-1,0\}$, we have that $\alpha$ is differentiable with
\begin{align}
  \alpha'(a)
  &=\frac{a^2\log(1+a)-2((1+a)\log(1+a)-a)a}{a^4}\notag\\
  &=\frac{2a^2-a(2+a)\log(1+a)}{a^4}.
\end{align}
By the l'H{\^o}pital rule, we obtain
\begin{align}
  \lim_{a\to0}\alpha'(a)
  &=\lim_{a\to0}\frac{4a-2(1+a)\log(1+a)-a(2+a)(1+a)^{-1}}{4a^3}\notag\\
  &=\lim_{a\to0}\frac{a(2+a)(1+a)^{-2}-2\log(1+a)}{12a^2}\notag\\
  &=-\lim_{a\to0}\frac{2a(2+a)(1+a)^{-3}}{24a}
   =-\frac{1}{12}\lim_{a\to0}\frac{2+a}{(1+a)^3}
   =-\frac{1}{6},
\end{align}
so defining $\alpha'(0)=-\frac{1}{6}$, we obtain that $\alpha'$ is
a continuous mapping on $(-1,\infty)$, and as $\alpha'$ is the
derivative of $\alpha$ for $a\ge-1$ with $a\notin\{-1,0\}$, $\alpha'$
is also the derivative of $\alpha$ for $(1,\infty)$. In order to show
that $\alpha$ is strictly decreasing, it then suffices to show that
that $2a^2-a(2+a)\log(1+a)$ is negative for $a>-1$ with $a\neq0$. Now,
for $a\neq0$, note that
\begin{align}
  \frac{\df{}}{\df{a}}(2a-(2+a)\log(1+a))
  &=2-\log(1+a)-\frac{2+a}{1+a}\quad\textrm{ and }\\
  \frac{\df{}^2}{\df{a}^2}(2a-(2+a)\log(1+a)) &=\frac{1}{(1+a)^2}-\frac{1}{1+a}
  =-\frac{a}{(1+a)^2}.
\end{align}
From this, we conclude that $a\mapsto 2a-(2+a)\log(1+a)$
is positive for $-1<a<0$ and negative for $a>0$. Therefore, $a\mapsto 2a^2-a(2+a)\log(1+a)$ is negative for $a>-1$ with
$a\neq0$. As a consequence, $\alpha$ is strictly decreasing. As
$\beta(a)=\alpha(a)/(1+a)$, the results on $\beta$ follow from those
on $\alpha$.
\end{proof}

\begin{lemma}
\label{lemma:PoissonCompMart}
Let $N$ be a standard Poisson process, let $b$ and $\lambda$ be in $\RR$, and define $f_b(\lambda)=\exp(-\lambda)+\lambda(1+b)-1$. With
$L^b_t=\exp(-\lambda(N_t-(1+b)t)-tf_b(\lambda))$, $L^b$ is a nonnegative
martingale with respect to the filtration induced by $N$.
\end{lemma}
\begin{proof}
Let $\GGG_t=\sigma(N_s)_{s\le t}$. Fix $0\le s\le t$. As $N_t-N_s$ is
independent of $\GGG_s$ and follows a Poisson distribution with
parameter $t-s$, we obtain
\begin{align}
  E(\exp(-\lambda(N_t-N_s))|\GGG_s)
  &=\exp((t-s)(\exp(-\lambda)-1)),
\end{align}
which implies
\begin{align}
  E(L^b_t|\GGG_s)
  &=E(\exp(-\lambda(N_t-N_s))|\GGG_s)
    \exp(-\lambda N_s)\exp(\lambda(1+b)t-tf_b(\lambda))\notag\\
  &=\exp((t-s)(\exp(-\lambda)-1))\exp(-\lambda N_s)\exp(\lambda(1+b)t-tf_b(\lambda))\notag\\
  &=\exp(-\lambda(N_s-(1+b)s)-sf_b(\lambda))=L^b_s,
\end{align}
proving the lemma.
\end{proof}

\begin{lemma}
\label{lemma:EMUI}
Let $M$ be a local martingale with initial value zero and $\Delta
M\ge-1$. Then $E\EEE(M)_\infty\le 1$, and $\EEE(M)$ is a uniformly
integrable martingale if and only if $E\EEE(M)_\infty=1$.
\end{lemma}
\begin{proof}
This follows from the the optional sampling theorem for nonnegative
supermartingales.
\end{proof}

In the proof of Theorem \ref{theorem:OptimalNovikov}, note that for a
standard Poisson process $N$, it holds that with $M_t = N_t-t$,
$\langle M\rangle_t = t$, since $[M]_t = N_t$ by Definition VI.37.6 of
\cite{RW2} and since $\langle M\rangle$ is the unique predictable and
locally integrable increasing process making $[M]-\langle M\rangle$ a
local martingale.

\begin{theorem}
\label{theorem:OptimalNovikov}
Fix $a\ge -1$. Let $M$ be a locally square integrable local martingale
with $\Delta
M1_{(\Delta M\neq0)}\ge a$. If $\exp(\frac{1}{2}\langle
M^c\rangle_\infty+\alpha(a)\langle M^d\rangle_\infty)$ is integrable,
then $\EEE(M)$ is a uniformly integrable martingale. Furthermore, for all
$a\ge-1$, the coefficients $\frac{1}{2}$ and $\alpha(a)$ in front of
$\langle M^c\rangle$ and $\langle M^d\rangle$ are optimal in the sense
that the criterion is false if any of the coefficients are reduced.
\end{theorem}
\begin{proof}
\textit{Sufficiency.} With $h(x)=(1+x)\log(1+x)-x$, we find by Lemma \ref{lemma:alphaFun} that
for $-1\le a\le x$, $\alpha(a)\ge \alpha(x)$, which implies $h(x)\le
\alpha(a)x^2$. Letting $a\ge-1$ and letting $M$ be a locally square
integrable local martingale with initial value zero, $\Delta M1_{(\Delta M\neq0)}\ge a$ and $\exp(\frac{1}{2}\langle
M^c\rangle_\infty+\alpha(a)\langle M^d\rangle_\infty)$ integrable, we
obtain for all $t\ge0$ the inequality $(1+\Delta M_t)\log(1+\Delta
M_t)+\Delta M_t\le \alpha(a) (\Delta M_t)^2$, and so Theorem III.1 of
\cite{LM} shows that $\EEE(M)$ is a
uniformly integrable martingale. Thus, the condition is sufficient. 

As regards optimality of the coefficients, optimality of the
coefficient $\frac{1}{2}$ in front of $\langle M^c\rangle$ is
well-known, see \cite{AAN}. It therefore suffices to to prove optimality of the
coefficient $\alpha(a)$ in front of $\langle M^d\rangle$. To do so, we
need to show the following: That
for each $\eps>0$, there exists a locally square
integrable local martingale with initial value zero and $\Delta M\ge a$ such that
$\exp(\frac{1}{2}\langle M^c\rangle_\infty+(1-\eps)\alpha(a)\langle M^d\rangle_\infty)$ is
integrable, while $\EEE(M)$ is not a uniformly integrable martingale.

\textit{The case $a>0$.} Let $\eps,b>0$, put $T_b = \inf\{t\ge0\mid N_t -
(1+b)t=-1\}$ and define $M_t = a(N^{T_b}_t-t\land T_b)$. We claim that
we may choose $b>0$ such that $M$ satisfies the requirements stated
above. It holds that $M$ is a locally square
integrable local martingale with initial value zero and $\Delta M1_{(\Delta M\neq0)}\ge a$, and $M$ is purely
discontinuous by Definition 7.21 of \cite{HWY} since it is of locally
integrable variation. In particular, $M^c=0$, so it suffices to show that $\exp((1-\eps)\alpha(a)\langle M\rangle_\infty)$ is
integrable while $\EEE(M)$ is not a uniformly
integrable martingale. To show this, we first argue that $T_b$ is almost surely finite. To this end, note that since
$t\mapsto N_t-(1+b)t$ only has nonnegative jumps, has initial value
zero and decreases between jumps, the process hits $-1$ if and only if it is less than or equal to
$-1$ immediately before one of its jumps. Therefore, with $U_n$ denoting the $n$'th jump time of
$N$, we have
\begin{align}
    P(T_b=\infty)
  &=P(\cap_{n=1}^\infty (N_{U_n-}-(1+b)U_n>-1))\notag\\
  &=P(\cap_{n=1}^\infty (n>U_n(1+b)))\notag\\
  &\le P(\limsup_{n\to\infty}U_n/n\le (1+b)^{-1}),
\end{align}
which is zero, as $\lim_{n\to\infty}U_n/n=1$ almost surely by the law
of large numbers, and $(1+b)^{-1}<1$ as $b>0$. Thus, $T_b$ is almost surely
finite, and by the path properties of $N$, $N_{T_b} = (1+b)T_b-1$
almost surely. We then obtain
\begin{align}
  \EEE(M)_\infty
  &=\exp(a(N_{T_b}-T_b)+N_{T_b}(\log(1+a)-a))\notag\\
  &=\exp(N_{T_b}\log (1+a)-aT_b)\notag\\
  &=\exp(((1+b)T_b-1)\log (1+a)-aT_b)\notag\\
  &=(1+a)^{-1}\exp(T_b((1+b)\log(1+a)-a)).
\end{align}
Recalling Lemma \ref{lemma:EMUI}, we wish to choose $b>0$ such that $E\EEE(M)_\infty<1$ and $E\exp((1-\eps)\alpha\langle
M\rangle_\infty)<\infty$ holds simultaneously. Note that $\langle
M\rangle_\infty = a^2T_b$. Therefore, we need to select a positive $b$
with the properties that
\begin{align}
    E\exp(T_b((1+b)\log(1+a)-a))&< 1+a\textrm{ and }\label{eqn:PrimaryFirstReq}\\
    E\exp(T_ba^2(1-\eps)\alpha(a))&<\infty.\label{eqn:PrimarySecondReq}
\end{align}
Consider some $b>0$ and let $f_b$ be as in Lemma \ref{lemma:PoissonCompMart}. By that same lemma, the process $L^b$ defined by
putting $L^b_t=\exp(-\lambda(N_t-(1+b)t)-tf_b(\lambda))$ is a
martingale. In particular, it is a nonnegative supermartingale with
initial value one, so Theorem II.77.5 of \cite{RW1} yields $1\ge EL^b_{T_b}
     =E\exp(\lambda-T_bf_b(\lambda))$,
and so $E\exp(-T_bf_b(\lambda))\le \exp(-\lambda)$. Note that $f_b'(\lambda)
  =-\exp(-\lambda)+1+b$, such that $f_b$ takes its minimum at $-\log(1+b)$. Therefore, $-f_b$ takes
its maximum at $-\log(1+b)$, and we find that the maximum is
$h(b)$. In particular, $E\exp(T_bh(b))$ is
finite. Next, define a function $\lambda$ by putting $\lambda(b) = -\log((1+a)\frac{b}{a})$, we then have $E\exp(-T_bf_b(\lambda(b)))\le
(1+a)\frac{b}{a}$, which is strictly less than $1+a$ whenever $b<a$. Thus, if we can choose $b\in(0,a)$ such that
\begin{align}
  (1+b)\log(1+a)-a&\le -f_b(\lambda(b))\textrm{ and }\label{eqn:FirstReq}\\
  a^2(1-\eps)\alpha(a)&\le h(b),\label{eqn:SecondReq}
\end{align}
we will have achieved our end, since (\ref{eqn:FirstReq}) implies
(\ref{eqn:PrimaryFirstReq}) and (\ref{eqn:SecondReq}) implies
(\ref{eqn:PrimarySecondReq}). To this end, note that
\begin{align}
  -f_b(\lambda(b))
  &=-\exp(\log((1+a)\tfrac{b}{a}))+\log((1+a)\tfrac{b}{a})(1+b)+1\notag\\
  &=1-(1+a)\tfrac{b}{a}+(1+b)\log((1+a)\tfrac{b}{a})\notag\\
  &=1-(1+a)\tfrac{b}{a}+(1+b)\log(1+a)+(1+b)\log \tfrac{b}{a},
\end{align}
such that, by rearrangement, (\ref{eqn:FirstReq}) is equivalent to
\begin{align}
  0\le 1+a-(1+a)\tfrac{b}{a}+(1+b)\log \tfrac{b}{a},
\end{align}
and therefore, as $1-\frac{b}{a}>0$ for $0<b<a$, equivalent to
\begin{align}
  \label{eqn:FirstReqReduction}
  (1+b)\frac{\log \tfrac{b}{a}}{\tfrac{b}{a}-1} &\le 1+a,
\end{align}
which, as $\log x\le x-1$ for $x>0$, is satisfied for all
$0<b<a$. Thus, it suffices to choose $b\in(0,a)$ such that
(\ref{eqn:SecondReq}) is satisfied, corresponding to choosing $b\in(0,a)$
such that $(1-\eps)h(a)\le h(b)$. As $h$ is positive and continuous on $(0,\infty)$, this is
possible by choosing $b$ close enough to $a$. With this choice of $b$, we now obtain $M$ yielding an
example proving that the coefficient $\alpha(a)$ is optimal. This
concludes the proof of optimality in the case $a>0$.

\textit{The case $a=0$.} Let $\eps>0$. To prove optimality, we wish to
identify a locally square integrable local martingale
$M$ with initial value zero and $\Delta M1_{(\Delta M\neq0)}\ge0$ such that $\exp((1-\eps)\alpha(0)\langle M\rangle_\infty)$ is integrable while
$\EEE(M)$ is not a uniformly integrable martingale. Recalling that
$\alpha$ is positive and continuous, pick $a>0$ so
close to zero that $(1-\eps)\alpha(0)\le
(1-\frac{1}{2}\eps)\alpha(a)$. By what was already shown, there
exists a locally square integrable local
martingale $M$ with initial value zero and $\Delta
M1_{(\Delta M\neq0)}\ge a$ such that $\exp((1-\frac{1}{2}\eps)\alpha(a)\langle
M\rangle_\infty)$ is integrable while $\EEE(M)$ is not a uniformly
integrable martingale. As $\exp((1-\eps)\alpha(0)\langle
M\rangle_\infty)$ is integrable in this case, this shows that $\alpha(0)$ is optimal.

\textit{The case $-1<a<0$.} Let $\eps>0$, let $-1<b<0$, let $c>0$ and
define a stopping time $T_{bc}$ by putting $T_{bc} = \inf\{t\ge0\mid N_t -
(1+b)t\ge c\}$. Also define $M$ by $M_t = a(N^{T_{bc}}_t-t\land T_{bc})$. We claim that
we can choose $b\in(-1,0)$ and $c>0$ such that $M$ satisfies the
requirements to show optimality. Similarly to the case $a>0$, $M$ is a
purely discontinuous locally square
integrable local martingale with initial value zero and $\Delta M1_{(\Delta M\neq0)}\ge a$, so it suffices to
show that $\exp((1-\eps)\alpha(a)\langle M\rangle_\infty)$ is
integrable while $\EEE(M)$ is not a uniformly integrable
martingale. We first investigate some properties
of $T_{bc}$. As $t\mapsto N_t-(1+b)t$ only has nonnegative jumps, has initial value
zero and decreases between jumps, the process advances beyond $c$ at
some point if and only it advances beyond $c$ at one of its jump
times. Therefore, with $U_n$ denoting the $n$'th jump time of $N$,
\begin{align}
    P(T_{bc}=\infty)
  &=P(\cap_{n=1}^\infty (N_{U_n}-(1+b)U_n<c))\notag\\
  &=P(\cap_{n=1}^\infty (n-c<U_n(1+b)))\notag\\
  &\le P(\liminf_{n\to\infty}U_n/n\ge (1+b)^{-1}),
\end{align}
which is zero, as $U_n/n$ tends to one almost surely and $(1+b)^{-1}>1$. Thus, $T_{bc}$ is almost surely
finite. Furthermore, by the path properties of $N$, $N_{T_{bc}}
\ge (1+b)T_{bc}+ c$ and $N_{T_{bc}}\le (1+b)T_{bc}+c+1$ almost surely. Since $\log(1+a)\le
0$, we in particular obtain $N_{T_{bc}}\log(1+a)\le ((1+b)T_{bc}+c)\log(1+a)$ almost
surely. From this, we conclude that
\begin{align}
  \EEE(M)_\infty
  &=\exp(a(N_{T_{bc}}-T_{bc})+N_{T_{bc}}(\log(1+a)-a))\notag\\
  &=\exp(N_{T_{bc}}\log (1+a)-aT_{bc})\notag\\
  &\le\exp(((1+b)T_{bc}+c)\log (1+a)-aT_{bc})\notag\\
  &=(1+a)^c\exp(T_{bc}((1+b)\log(1+a)-a)).
\end{align}
We wish to choose $-1<b<0$ and $c>0$ such that $E\exp((1-\eps)\alpha(a)\langle
M\rangle_\infty)<\infty$ and $E\EEE(M)_\infty<1$ holds
simultaneously. As $\langle M\rangle_\infty = a^2T_{bc}$, this is
equivalent to choosing $-1<b<0$ and $c>0$ such that
\begin{align}
    E\exp(T_{bc}((1+b)\log(1+a)-a))&< (1+a)^{-c}\textrm{ and }\label{eqn:PrimaryFirstReqMinus}\\
    E\exp(T_{bc}a^2(1-\eps)\alpha(a))&<\infty.\label{eqn:PrimarySecondReqMinus}
\end{align}
Let $f_b$ and $L^b$ be as in Lemma \ref{lemma:PoissonCompMart}. The
process $L^b$ is then a
nonnegative supermartingale. As $N_{T_{bc}}\le (1+b)T_{bc}+c+1$, the
optional stopping theorem allows us to conclude that for $\lambda\ge0$,
\begin{align}
  1&\ge EL^b_{T_{bc}}
   =E\exp(-\lambda(N_{T_{bc}}-(1+b)T_{bc})-T_{bc}f_b(\lambda))\notag\\
  &\ge E\exp(-(c+1)\lambda-T_{bc}f_b(\lambda)),
\end{align}
so that $E\exp(-T_{bc}f_b(\lambda))\le \exp((c+1)\lambda)$. As in the case $a>0$, $-f_b$ takes
its maximum at $-\log(1+b)$, and the maximum is $h(b)$, leading
us to conclude that $E\exp(T_{bc}h(b))$ is finite. Put
$\lambda(b,c)=(c+1)^{-1}\log((1+a)^{-c}\frac{b}{a})$. For all $b\in(a,0)$,
 $\frac{b}{a}<1$, leading to $E\exp(-T_{bc}f_b(\lambda(b,c)))\le
(1+a)^{-c}\frac{b}{a}<(1+a)^{-c}$. Therefore, if we can choose
$b\in(a,0)$ and $c>0$ such that
\begin{align}
  (1+b)\log(1+a)-a&\le -f_b(\lambda(b,c)) \textrm{ and }\label{eqn:FirstReqMinus}\\
  a^2(1-\eps)\alpha(a)&\le h(b),\label{eqn:SecondReqMinus}
\end{align}
we will have obtained existence of a local maringale yielding the
desired optimality of $\alpha(a)$. We first note that $a^2(1-\eps)\alpha(a)\le h(b)$
is equivalent to $(1-\eps)h(a)\le h(b)$. As $h$ is continuous and
positive on $(-1,0)$, we find that (\ref{eqn:SecondReqMinus}) is satisfied
for $a<b<0$ with $b$ close enough to $a$. Next, we turn our
attention to (\ref{eqn:FirstReqMinus}). We have
\begin{align}
  -f_b(\lambda(b,c))
  &=-\exp\left(-\frac{1}{c+1}\log\left((1+a)^{-c}\frac{b}{a}\right)\right)-\frac{1+b}{c+1}\log\left((1+a)^{-c}\frac{b}{a}\right)+1\notag\\
  &=1-(1+a)^{\frac{c}{c+1}}\left(\frac{b}{a}\right)^{-\frac{1}{(c+1)}}-\frac{1+b}{c+1}\log\left((1+a)^{-c}\frac{b}{a}\right)\notag\\
  &=1-(1+a)^{\frac{c}{c+1}}\left(\frac{a}{b}\right)^{\frac{1}{c+1}}+\frac{c(1+b)}{c+1}\log(1+a)+\frac{1+b}{c+1}\log\frac{a}{b},
\end{align}
such that (\ref{eqn:FirstReqMinus}) is equivalent to
\begin{align}
  \label{eqn:FirstReqReductionMinus}
  0&\le
  1+a-(1+a)^{\frac{c}{c+1}}\left(\frac{a}{b}\right)^{\frac{1}{c+1}}+\frac{1+b}{c+1}\left(\log\frac{a}{b}-\log(1+a)\right).
\end{align}
Fixing $a<b<0$, we wish to argue that for $b$ close enough to $a$, 
(\ref{eqn:FirstReqReductionMinus}) holds for $c$ large enough. To this
end, let $\rho_b(c)$ denote the right-hand side of
(\ref{eqn:FirstReqReductionMinus}). Then
$\lim_{c\to\infty}\rho_b(c)=0$. We also note that
$\frac{\df{}}{\df{c}}\frac{1}{c+1}=-\frac{1}{(c+1)^2}$ and
$\frac{\df{}}{\df{c}}\frac{c}{c+1}=\frac{1}{(c+1)^2}$, yielding
\begin{align}
  &\frac{\df{}}{\df{c}}(1+a)^{\frac{c}{c+1}}\left(\frac{a}{b}\right)^{\frac{1}{c+1}}\notag\\
  &=\frac{\df{}}{\df{c}}\exp\left(\frac{c}{c+1}\log(1+a)+\frac{1}{c+1}\log\frac{a}{b}\right)\notag\\
  &=\left(\frac{\log(1+a)}{(c+1)^2}-\frac{\log\frac{a}{b}}{(c+1)^2}\right)\exp\left(\frac{c}{c+1}\log(1+a)+\frac{1}{c+1}\log\frac{a}{b}\right)\notag\\
  &=\frac{\log(1+a)-\log\frac{a}{b}}{(c+1)^2}\exp\left(\frac{c}{c+1}\log(1+a)+\frac{1}{c+1}\log\frac{a}{b}\right)
\end{align}
and
\begin{align}
  \frac{\df{}}{\df{c}}\frac{1+b}{c+1}\left(\log\frac{a}{b}-\log(1+a)\right)
  &=-\frac{1+b}{(c+1)^2}\left(\log\frac{a}{b}-\log(1+a)\right)\notag\\
  &=(1+b)\frac{\log(1+a)-\log\frac{a}{b}}{(c+1)^2},
\end{align}
which leads to
\begin{align}
  \rho_b'(c)
  &=\frac{\log(1+a)-\log\frac{a}{b}}{(c+1)^2}\left(1+b-\exp\left(\frac{c}{c+1}\log(1+a)+\frac{1}{c+1}\log\frac{a}{b}\right)\right).
\end{align}
Now note that for $a<b$, we obtain
\begin{align}
  \lim_{c\to\infty}1+b-\exp\left(\frac{c}{c+1}\log(1+a)+\frac{1}{c+1}\log\frac{a}{b}\right)
  &=1+b-(1+a)>0,
\end{align}
and for $b$ close enough to $a$, $\log(1+a)-\log\frac{a}{b}<0$,
since $a<0$. Therefore, for all $c$ large enough, $\rho_b'(c)<0$. Consider such
a $c$, we then obtain
\begin{align}
  \rho_b(c)&=\lim_{y\to\infty}\rho_b(c)-\rho_b(y)
       =-\lim_{y\to\infty}\int_c^y \rho_b'(z)\df{z}>0.
\end{align}
Thus, we conclude that for $b$ close enough to $a$, it holds that $\rho_b(c)>0$ for $c$ large enough.

We now collect our conclusions in order to obtain $b\in(a,0)$ and $c>0$
satisfying (\ref{eqn:FirstReqMinus}) and
(\ref{eqn:SecondReqMinus}). First choose $b\in(a,0)$ so close to
$a$ that $(1-\eps)h(a)\le h(b)$ and
$\log(1+a)-\log\frac{a}{b}<0$. Pick $c$ so large that $\rho_b(c)>0$. By our deliberations, (\ref{eqn:FirstReqMinus}) and
(\ref{eqn:SecondReqMinus}) then both hold, demonstrating the existence of
a locally square integrable local martingale $M$ with initial value
zero and $\Delta M1_{(\Delta M\neq0)}\ge a$ such that $\exp((1-\eps)\alpha(a)\langle M\rangle_\infty)$ is
integrable while $\EEE(M)$ is not a uniformly integrable
martingale.

\textit{The case $a=-1$.} Let $\eps>0$. We wish to identify a purely
discontinous locally square integrable local martingale $M$ with
$\Delta M1_{(\Delta M\neq0)}\ge-1$ such that integrability of
$\exp((1-\eps)\alpha(-1)\langle M\rangle_\infty)$ holds while
$\EEE(M)$ is not a uniformly integrable martingale. We proceed as in
the case $a=0$. By positivity and continuity of $\alpha$, take $a>0$ so
close to $-1$ that $(1-\eps)\alpha(-1)\le
(1-\frac{1}{2}\eps)\alpha(a)$. By what was shown in the previous case, there
exists a purely discontinuous locally square integrable local
martingale $M$ with initial value zero and $\Delta
M1_{(\Delta M\neq0)}\ge a$ such that $\exp((1-\frac{1}{2}\eps)\alpha(a)\langle
M\rangle_\infty)$ is integrable while $\EEE(M)$ is not a uniformly
integrable martingale. As $\exp((1-\eps)\alpha(-1)\langle
M\rangle_\infty)$ then also is integrable, this shows that $\alpha(-1)$ is optimal.
\end{proof}

\begin{theorem}
\label{theorem:OptimalNovikov2}
Fix $a>-1$. Let $M$ be a local martingale with $\Delta M1_{(\Delta M\neq0)}\ge a$. If $\exp(\frac{1}{2}[M^c]_\infty+\beta(a)[M^d]_\infty)$ is integrable,
then $\EEE(M)$ is a uniformly integrable martingale. Furthermore, for all
$a>-1$, the coefficients $\frac{1}{2}$ and $\beta(a)$ in front of
$[M^c]$ and $[M^d]$ are optimal in the sense
that the criterion is false if any of the coefficients are
reduced.

Furthermore, there exists no $\beta(-1)$ such that for $M$ with $\Delta M1_{(\Delta
  M\neq0)}\ge-1$, integrability of $\exp(\frac{1}{2}[M^c]_\infty+\beta(-1)[M^d]_\infty)$ suffices to
ensure that $\EEE(M)$ is a uniformly integrable martingale.
\end{theorem}
\begin{proof}
\textit{Sufficiency.} We proceed in a manner closely related to the
proof of Theorem \ref{theorem:OptimalNovikov}. Defining $g$ by putting
$g(x)=\log(1+x)-x/(1+x)$, we find by Lemma \ref{lemma:alphaFun} that
for $-1<a\le x$, $\beta(a)\ge \beta(x)$, yielding $g(x)\le
\beta(a)x^2$. Letting $a>-1$ and letting $M$ be a locally square
integrable local martingale with initial value zero, $\Delta M1_{(\Delta M\neq0)}\ge a$ and $\exp(\frac{1}{2}\langle
M^c\rangle_\infty+\beta(a)\langle M^d\rangle_\infty)$ integrable, we
obtain for all $t\ge0$ that $\log(1+\Delta M_t)+\Delta
M_t/(1+\Delta M_t)\le \beta(a) (\Delta M_t)^2$, and so Theorem III.7
of \cite{LM} shows that $\EEE(M)$ is a uniformly integrable
martingale. Thus, the condition is sufficient. 

As in Theorem \ref{theorem:OptimalNovikov}, optimality of the $\frac{1}{2}$ in front of $[M^c]$ follows from \cite{AAN},
so it suffices to consider the coefficient $\beta(a)$ in front of
$[M^d]$. Thus, for $a>-1$, we need to prove that
for each $\eps>0$, there exists a locally square
integrable local martingale with initial value zero and $\Delta
M1_{(\Delta M\neq0)}\ge a$ such that
$\exp(\frac{1}{2}[M^c]_\infty+(1-\eps)\beta(a)[M^d]_\infty)$ is
integrable, while $\EEE(M)$ is not a uniformly integrable martingale.

\textit{The case $a>0$.} Let $\eps,b>0$, put $T_b = \inf\{t\ge0\mid N_t -
(1+b)t=-1\}$ and define $M_t = a(N^{T_b}_t-t\land T_b)$. Noting that
$[M]_\infty = a^2N_{T_b}$, we may argue as in the proof of Theorem
\ref{theorem:OptimalNovikov} and obtain that it suffices to identify
$b>0$ such that
\begin{align}
    E\exp(T_b((1+b)\log(1+a)-a))&< 1+a\textrm{ and }\label{eqn:SqPrimaryFirstReq}\\
    E\exp(N_{T_b}a^2(1-\eps)\beta(a))&<\infty.\label{eqn:SqPrimarySecondReq}
\end{align}
Let $f_b$ be as in Lemma \ref{lemma:PoissonCompMart}. As in the proof of Theorem
\ref{theorem:OptimalNovikov}, we obtain that $E\exp(T_bh(b))$ is
finite, where $h(x)=(1+x)\log(1+x)-x$, and furthermore obtain that with $\lambda(b) = -\log((1+a)\frac{b}{a})$,
$E\exp(-T_bf_b(\lambda(b)))<1+a$ for $b<a$. As $N_{T_b}=(1+b)T_b-1$
almost surely and $g(b)=h(b)/(1+b)$, we
then also obtain that $E\exp(N_{T_b}g(b))$ is finite. Thus, if we can choose $b\in(0,a)$ such that
\begin{align}
  (1+b)\log(1+a)-a&\le -f_b(\lambda(b))\textrm{ and }\label{eqn:SqFirstReq}\\
  a^2(1-\eps)\beta(a)&\le g(b),\label{eqn:SqSecondReq}
\end{align}
we will obtain the desired result, as (\ref{eqn:SqFirstReq}) implies
(\ref{eqn:SqPrimaryFirstReq}) and (\ref{eqn:SqSecondReq}) implies
(\ref{eqn:SqPrimarySecondReq}). As earlier noted,
(\ref{eqn:SqFirstReq}) always holds for $0<b<a$. As for
(\ref{eqn:SqSecondReq}), this requirement is equivalent to having that
$(1-\eps)g(a)\le g(b)$ for some $b\in (0,a)$, which by continuity of
$g$ can be obtained by choosing $b$ close enough to $a$. Choosing $b$
in this manner, we obtain $M$ yielding an example proving that the
coefficient $\beta(a)$ is optimal. This concludes the proof of
optimality in the case $a>0$.

\textit{The case $a=0$.} This follows similarly to the corresponding
case in the proof of Theorem \ref{theorem:OptimalNovikov}.

\textit{The case $-1<a<0$.} Let $\eps>0$, let $-1<b<0$, let $c>0$ and
define a stopping time $T_{bc}$ by putting $T_{bc} = \inf\{t\ge0\mid N_t -
(1+b)t\ge c\}$. Also define $M$ by $M_t = a(N^{T_{bc}}_t-t\land
T_{bc})$. As in the proof of Theorem \ref{theorem:OptimalNovikov}, in
order to obtain the desired counterexample, it suffices to choose
$-1<b<0$ and $c>0$ such that
\begin{align}
    E\exp(T_{bc}((1+b)\log(1+a)-a))&< (1+a)^{-c}\textrm{ and }\label{eqn:SqPrimaryFirstReqMinus}\\
    E\exp(T_{bc}a^2(1-\eps)\beta(a))&<\infty.\label{eqn:SqPrimarySecondReqMinus}
\end{align}
With $f_b$ as in Lemma \ref{lemma:PoissonCompMart}, we find as in the
proof of Theorem \ref{theorem:OptimalNovikov} that $E\exp(T_{bc}h(b))$
is finite. Furthermore, defining
$\lambda(b,c)=(c+1)^{-1}\log((1+a)^{-c}\frac{b}{a})$, it holds for $b$
with $a<b\le (1+a)^ca$ that $\lambda(b,c)\ge0$ and
$E\exp(-T_{bc}f_b(\lambda(b,c)))<(1+a)^{-c}$. Also, as
$N_{T_{bc}}\le(1+b)T_{bc}+c+1$, $E\exp(N_{T_{bc}}(1+b)^{-1}h(b))$ and
thus $E\exp(N_{T_{bc}}g(b))$ is finite. Therefore, if we can choose
$b\in(a,0)$ and $c>0$ such that 
\begin{align}
  (1+b)\log(1+a)-a&\le -f_b(\lambda(b,c)) \textrm{ and }\label{eqn:SqFirstReqMinus}\\
  a^2(1-\eps)\beta(a)&\le g(b),\label{eqn:SqSecondReqMinus}
\end{align}
we obtain the desired result. By arguments as in the proof of the corresponding case of Theorem
\ref{theorem:OptimalNovikov}, we find that by first picking $b$ close
enough to $a$ and then $c$ large enough, we can ensure that both (\ref{eqn:SqFirstReqMinus}) and
(\ref{eqn:SqSecondReqMinus}) hold, yielding optimality for this case.

\textit{The case $a=-1$.} For this case, we need to show that for any
$\gamma\ge0$, it does not hold that finiteness of $E\exp(\gamma[M^d]_\infty)$
implies that $\EEE(M)$ is a uniformly integrable martingale. Let
$\gamma\ge0$. By Lemma \ref{lemma:alphaFun}, $\beta(a)$ tends to infinity as
$a$ tends to $-1$. Therefore, we
may pick $a>-1$ so small that $\beta(a)\ge \gamma$. By what we
already have shown, there exists $M$ with initial value zero and
$\Delta M1_{(\Delta M\neq0)}\ge -1$ such that
$E\exp(\beta(a)[M^d]_\infty)$ and thus $E\exp(\gamma[M^d]_\infty)$ is finite,
while $\EEE(M)$ is not a uniformly integrable martingale.
\end{proof}

\begin{corollary}
\label{coro:NovikovExtension}
Let $M$ be a local martingale with initial value zero and $\Delta
M\ge0$. If $\exp(\frac{1}{2}[M]_\infty)$ is integrable or if $M$ is
locally square integrable and $\exp(\frac{1}{2}\langle
M\rangle_\infty)$ is integrable, then $\EEE(M)$ is a uniformly
integrable martingale. Furthermore, this criterion is optimal in the
sense that if either the constant $\frac{1}{2}$ is reduced, or the
requirement on the jumps is weakened to $\Delta M\ge-\eps$ for some
$\eps>0$, the criterion ceases to be sufficient.
\end{corollary}

\begin{proof}
That the constant $\frac{1}{2}$ cannot be reduced follows from
Theorem \ref{theorem:OptimalNovikov} and Theorem
\ref{theorem:OptimalNovikov2}. That the requirement on the jumps cannot
be reduced follows by combining Theorem \ref{theorem:OptimalNovikov} and Theorem
\ref{theorem:OptimalNovikov2} with the fact that $\alpha$ and $\beta$ both are strictly
decreasing by Lemma \ref{lemma:alphaFun}.
\end{proof}

\bibliographystyle{amsplain}

\begin{thebibliography}{99}

\bibitem{TB} Bj{\"o}rk, T.: {\em Arbitrage theory in continuous
    time}, 3rd edition, Oxford University press, 2009.

\bibitem{HWY} He, S.-W., Wang, J.-G. \& Yan, J.-A.:
  Semimartingale Theory and Stochastic Calculus, Science Press, CRC
  Press Inc., 1992.

\bibitem{ISS} Izumisawa, M., Sekiguchi, T., and Shiota, Y.: Remark on a characterization of
BMO-martingales. {\em T{\^o}hoku Math. J.} {\bf 31} (3) (1979) 281--284.

\bibitem{KS} Kallsen, J. and Shiryaev, A. N.: The cumulant process and Esscher's change of
measure. {\em Finance Stoch.} {\bf 6} (4) (2002), 397--428.

\bibitem{KS2} Karatzas, I. and S. E. Shreve: {\em Brownian Motion and
    Stochastic Calculus}, Springer-Verlag, 1988.

\bibitem{KA} Karr, A.: {\em Point Processes and their Statistical
    Inference}, Marcel Dekker, Inc., 1986.

\bibitem{YK} Kutoyants, Y. A..: {\em Statistical inference for Ergodic
  Diffusion Processes}, Springer-Verlag, 2004.

\bibitem{LM} L{\'e}pingle, D. and M{\'e}min, J.: Sur
  l'int{\'e}grabilit{\'e} uniforme des martingales
  exponentielles. {\em Z. Wahrsch. Verw. Gebiete} {\bf 42}(3) (1978)
  175--203.

\bibitem{AAN} Novikov, A. A.: On an identity for stochastic
integrals, {\em Theor. Probability Appl.} {\bf 17} (1972) 717--720.

\bibitem{TO} Okada, T.: A criterion for uniform integrability of
  exponential martingales. {\em T{\^o}hoku Math. J.} {\bf 34} (4) (1982) 495--498.

\bibitem{PS} Protter, P. E. and Shimbo, K.: No Arbitrage and General Semimartingales. In ``A festschrift for Thomas G. Kurtz'', 2008.

\bibitem{PP} Protter, P.: {\em Stochastic Integration and Differential
  Equations}, 2nd edition, Springer, 2005.

\bibitem{RW1} Rogers, L. C. G. and Williams, D.: {\em Diffusions, Markov
  Processes and Martingales}, volume 1, Cambridge University Press, 2000.

\bibitem{RW2} Rogers, L. C. G. and Williams, D.: {\em Diffusions, Markov
  Processes and Martingales}, volume 2, Cambridge University Press, 2000.

\bibitem{JAY} Yan, J.-A.: {\`A} propos de l'int{\'e}grabilit{\'e} uniforme des
  martingales exponentielles. {\em S{\'e}minaire de probabilit{\'e}s (Strasbourg)} {\bf 16} (1982) 338--347.

\end{thebibliography}

\end{document}